\documentclass[11pt, reqno]{amsart}
\usepackage{amssymb, amsthm, amsmath, amsfonts}
\usepackage{graphics}
\usepackage{hyperref}
\usepackage{bbm}
\usepackage[all]{xy}
\usepackage{enumerate}
\usepackage[mathscr]{eucal}
\usepackage{enumerate}
\usepackage{thmtools}
\usepackage{xcolor}
\usepackage{ytableau}

\theoremstyle{plain}
\newtheorem{theorem}{Theorem}[section]
\newtheorem{lemma}[theorem]{Lemma}
\newtheorem{proposition}[theorem]{Proposition}

\newtheorem{corollary}[theorem]{Corollary}
\numberwithin{equation}{section}

\theoremstyle{definition}

\newtheorem{definition}[theorem]{Definition}
\newtheorem{example}[theorem]{Example}
\newtheorem{remark}[theorem]{Remark}

\newcommand{\Mod}{{\textrm{-}\mathrm{Mod}}}
\newcommand{\dis}{{\mathrm{dis}}}

\DeclareMathOperator{\Ob}{Ob}
\DeclareMathOperator{\Aut}{Aut}
\DeclareMathOperator{\Sym}{Sym}

\DeclareMathOperator{\Hom}{Hom}

\DeclareMathOperator{\Stab}{Stab}

\DeclareMathOperator{\Age}{Age}

\DeclareMathOperator{\Res}{Res}

\newcommand{\C}{{\mathscr{C}}}
\newcommand{\D}{{\mathscr{D}}}
\newcommand{\CI}{{\mathrm{CI}}}
\newcommand{\OI}{{\mathrm{OI}}}
\newcommand{\FI}{{\mathrm{FI}}}
\newcommand{\BI}{{\mathrm{BI}}}
\newcommand{\SI}{{\mathrm{SI}}}

\newcommand{\op}{{\mathrm{op}}}

\newcommand{\Sh}{{\mathrm{Sh}}}

\newcommand{\PSh}{{\mathrm{PSh}}}

\title[Representations of permutation groups]{On representations of permutation groups and their orbit categories}

\author{Liping Li}
\address{School of Mathematics and Statistics, Hunan Normal University, Changsha 410081, China.}
\email{lipingli@hunnu.edu.cn}

\thanks{L. Li was partly supported by NSFC Grant No. 12171146.}

\keywords{Infinite permutation groups, polynomial rings, discrete representations, orbit categories.}

\begin{document}

\begin{abstract}
Given an infinite set $\Omega$ and a ring $R$ as well as a group $G$ acting on them, we show that $G$ and a subgroup $H$ share the same canonical relational structure on $\Omega$ if and only if the restriction functor gives an equivalence from the category of discrete representations of $G$ to that of $H$. Moreover, the age of this relational structure satisfies the strong amalgamation property if and only if there is a canonical isomorphism from the category of finite substructures of $\Omega$ and embeddings to the opposite category of the orbit category of $G$. As an application, we prove that finitely generated discrete representations of highly homogeneous groups over the polynomial ring $k[\Omega]$ are Noetherian.
\end{abstract}

\maketitle

\section{Introduction}

The study of infinite permutation groups occupies a central position at the intersection of group theory, combinatorics, and model theory. They are not only closely tied to combinatorial notions such as homogeneity, oligomorphy, and Ramsey-type properties, but also play a prominent role in model theory, where the automorphism group of a countable structure encodes substantial information about the definability and classification of the structure. For example, the correspondence between oligomorphic groups and $\omega$-categorical structures, developed by Engeler, Ryll-Nardzewski, and Svenonius, has been a driving force in connecting infinite group actions to logical classification problems \cite{Engeler60,Ryll60}.

In more recent years, representation theory of infinite permutation groups, in particular their actions on commutative rings emerged as a new complementary area. Applying techniques from commutative algebra, invariant theory, and algebraic geometry, people have obtained significant progress along this approach. One of the key achievements of this perspective is the development of \emph{equivariant Noetherianity}. Explicitly, let $G$ be a permutation group on an infinite set $\Omega$, and $k$ a commutative Noetherian ring. While the polynomial ring $k[\Omega]$ whose indeterminates are elements in $\Omega$ is far from Noetherian in the classical sense, it often exhibits Noetherian behavior once the symmetries of an infinite permutation group are taken into account. For instance, Cohen \cite{Co} proved that the poset of $G$-stable ideals of $R$ satisfies the ascending chain condition; Aschenbrenner and Hillar \cite{AH} showed that the ideal of relations among minors of an infinite generic matrix is finitely generated up to the natural action of the infinite general linear group. Subsequent works have shown that a broad class of rings and varieties have this property, which has had fruitful applications in algebra, geometry, and combinatorics such as conceptual explanation for the stabilization of syzygies in large algebraic structures, asymptotic behavior in algebraic statistics and tensor problems, etc. Moreover, this topic has deeply been intertwined with other areas such as representation theory of combinatorial categories. For a few references, see the works of Draisma \cite{Dr, Dr18}, Hillar-Sullivant \cite{HS}, Laudone-Snowden \cite{LS}, Van Le-Nagel-Nguyen-R\"{o}mer \cite{LNNR}, Li-Peng-Yuan \cite{LPY}, Nagel-R\"{o}mer \cite{NR}, Nagpal-Snowden \cite{NS}, Ramos \cite{Ramos}, Sam-Snowden \cite{SS1, SS2}, and many more.

In \cite{LPY} the author and his collaborators proved that $k[\Omega]$ satisfies the $G$-equivariant Noetherianity if the action of $G$ on $\Omega$ is highly homogeneous. Moreover, for two special cases, the category of discrete representations of $G$ over $k[\Omega]$ is equivalent to that of $\tilde{G}$ over $k[\Omega]$, where $\tilde{G}$ is the automorphism group of the canonical relational structure induced by the action of $G$. For details, see \cite[Theorems 1.1 and 1.3 and Corollary 5.16]{LPY}. Since $G$ and $\tilde{G}$ have the same growth rate, or equivalently they share the same orbits on each $\Omega^n$, a natural question arises as follows:

\vspace{2mm}
\textbf{Question.} Suppose that the action of $G$ on $\Omega$ is oligomorphic, and let $H$ be a subgroup of $G$. Is it true that their categories of discrete representations are equivalent if and only if their growth rates eventually coincide?\footnote{This question is raised by the reviewer and appears in \cite[Question 1.7]{LPY}.}
\vspace{2mm}

The first main result of this paper gives an affirmative answer to the above question in an even more general setup. Before describing it, let me describe necessary notions and definitions. Throughout this paper let $\Omega$ be an infinite set and $G$ a group acting on it. Impose on $G$ the \textit{permutation topology}, the topology such that stabilizer subgroups of finite subsets of $\Omega$ form a neighborhood basis of the identity. Let $R$ be a unital ring equipped with the discrete topology such that elements of $G$ act on it continuously as ring automorphisms. Denote by $RG$ the skew group ring. An $RG$-module $V$ is called \textit{discrete} if the action of $G$ on $V$ is continuous with respect to the discrete topology on $V$ and the permutation topology on $G$, or equivalently, for any $v \in V$, the stabilizer subgroup $\Stab_G(v)$ is open in $G$.

Let $RG \Mod^{\dis}$ be the category of discrete $RG$-modules, which is a Grothendieck category.

\begin{theorem} \label{thm 1}
If two groups $G$ and $H$ acting on $\Omega$ share the same orbits on $\Omega^n$ for infinitely many $n$, then we have an equivalence
\[
RG \Mod^{\dis} \simeq RH \Mod^{\dis}.
\]
Furthermore, if $H$ is a subgroup of $G$, then they share the same orbits on each $\Omega^n$ if and only if the restriction functor
\[
\Res: RG \Mod^{\dis} \longrightarrow RH \Mod^{\dis}
\]
induces the above equivalence.
\end{theorem}

\begin{remark}
We show that $G$ and $H$ share the same orbits on each $\Omega^n$ if and only if they share the same orbits on infinitely many $\Omega^n$'s; see Corollary \ref{orbits on cartesian products}. In the case that $H$ is a subgroup of $G$ and $G$ is oligomorphic, it is equivalent to the following condition: the growth rate of $H$ and the growth rate of $G$ coincide for infinitely many $n \in \mathbb{N}$; see Corollary \ref{growth functions}. Thus this theorem indeed answers the above question.
\end{remark}

\begin{remark}
Actually, if $H$ is a subgroup of $G$, then they share the same orbits if and only if $H$ is a dense subgroup of $G$; see Proposition \ref{topological characterization}. Thus the if direction of the second statement of this theorem is an analogue to the following result which might be well known to experts: if $G$ is a totally disconnected group equipped with a neighborhood basis of the identity consisting of open subgroups, and $H$ is a dense subgroup of $G$, then the restriction functor gives an equivalence from the category of smooth representations of $G$ over a field (which are exactly discrete representations in our sense) to that of $H$. We did not find a clear statement and its proof of this result in literature, but the reader can refer to \cite{JR} for more details.
\end{remark}

The \textit{orbit category} $\mathcal{G}$ of $G$ with respect to the family of pointwise stabilizer subgroups $G_{\Gamma}$ of finite subsets $\Gamma$ plays a fundamental role in the research of discrete representation of $G$. Indeed, by \cite[Theorem 2.2]{LPY}, the category of discrete representation of $G$ is equivalent to the category of sheaves of modules over a ringed site with $\mathcal{G}$ the underlying category. Thus people may use sheaf theory and representation theory of categories to consider discrete representations of topological groups; see for instance \cite{DLLX}. On the other hand, the action of $G$ on $\Omega$ induces a canonical relational structure $\mathcal{M} = (\Omega, (R_i)_{i \in I})$, where each $R_i \subseteq \Omega^{n_i}$ is a $G$-orbit in $\Omega^{n_i}$. Thus one obtains another category $\mathcal{A}$: the category of finite substructures of $\mathcal{M}$ and embeddings.

We note that the map $\phi$ sending each finite substructure $\Gamma$ to the left coset $G/G_{\Gamma}$ gives rise to a fully faithful functor (also denoted by $\phi$) from $\mathcal{A}^{\op}$ to $\mathcal{G}$; see the proof of Theorem \ref{orbit categories}. Furthermore, there are quite a few interesting examples for which $\phi$ is an isomorphism. For instance, in the case that the action of $G$ on $\Omega$ is highly transitive, then $\mathcal{G}^{\op}$ is isomorphic to the category of finite subsets of $\Omega$ and injections; for $G = \Aut(\mathbb{Q}, \leqslant)$, then $\mathcal{G}^{\op}$ is isomorphic to the category of finite subsets of $\mathbb{Q}$ and order-preserving injections. Thus we may ask the following question:

\vspace{2mm}
\textbf{Question.} Under what conditions is the functor $\phi$ an isomorphism?
\vspace{2mm}

It turns out the strong amalgamation property introduced in \cite[Section 2.7]{Ca} is the correct condition that guarantees this isomorphism.

\begin{theorem} \label{thm 2}
The following statements are equivalent:
\begin{enumerate}
\item the functor $\phi: \mathcal{A}^{\op} \to \mathcal{G}$ is an isomorphism;

\item $\Age(\mathcal{M})$ satisfies the strong amalgamation property;

\item for every finite subsets $\Gamma \subseteq \Omega$, the stabilizer subgroup $G_{\Gamma}$ has no fixed points outside $\Gamma$.
\end{enumerate}
\end{theorem}

As an application, we consider the case that the action of $G$ on $\Omega$ is highly homogeneous, $k$ a commutative Noetherian ring, and $R = k[\Omega]$. We show that $\Age(\mathcal{M})$ satisfies the strong amalgamation property. Consequently, by Cameron's classification of highly homogeneous groups \cite{Ca76}, the above theorem and \cite[Theorem 2.2]{LPY}, we deduce that
\[
RG \Mod^{\dis} \simeq \Sh(\mathcal{A}^{\op}, J_{at}, \mathcal{R})
\]
where $J_{at}$ is the atomic Grothendieck topology and $\mathcal{R}$ is the structure sheaf with $\mathcal{R}(\Gamma) = k[\Gamma]$ for every finite subset $\Gamma$ of $\Omega$. Relying on this equivalence and the technique in \cite{NR}, we obtain the following result, generalizing \cite[Theorem 6.15]{NR} and \cite[Corollary 1.5]{LPY} simultaneously.

\begin{theorem} \label{thm 3}
Let $k$ and $R$ be as defined above. If the action of $G$ on $\Omega$ is highly homogeneous, then every finitely generated presheaf over the ringed site $(\mathcal{A}^{\op}, J_{at}, \mathcal{R})$ is Noetherian. Consequently, every finitely generated discrete $RG$-module is Noetherian.
\end{theorem}

The paper is organized as follows. In Section 2 we deduce some elementary facts on actions of groups on infinite sets. In Section 3 we use standard arguments on representation theory of topological groups to prove Theorem \ref{thm 1}. Relations between the orbit category $\mathcal{G}$ and the category $\mathcal{A}$ of finite substructures and embeddings is explored in Section 4, where we prove Theorem \ref{thm 2}. Applications and a proof of Theorem \ref{thm 3} are given in the final section.

\section{Actions on sets and orbits}

 From now on let $\Omega$ be an infinite set, and $\Sym(\Omega)$ the full permutation group on $\Omega$. Given a group $G$ together with a fixed group homomorphism $\rho: G \to \Sym(\Omega)$, we can define an action of $G$ on $\Omega$ via
\[
G \times \Omega \longrightarrow \Omega, \quad (g, x) \mapsto g \cdot x = \rho(g)(x).
\]

For $n \in \mathbb{N}$, define
\begin{itemize}
\item $\Omega^n$: the Cartesian product of $n$ copies of $\Omega$,

\item $\Omega^{(n)}$: the subset of $\Omega^n$ consisting of componentwise distinct $n$-tuples,

\item $\Omega^{\{n\}}$: the set of $n$-element subsets of $\Omega$.
\end{itemize}
The action of $G$ on $\Omega$ naturally induces actions of $G$ on these sets, so we define the following numerical invariants:
\begin{align*}
& f_{G, \Omega}: \mathbb{N} \to \mathbb{N} \cup \{ \infty\}, \quad n \mapsto \text{ the number of $G$-orbits in } \Omega^{\{n\}},\\
& F_{G, \Omega}: \mathbb{N} \to \mathbb{N} \cup \{ \infty\}, \quad n \mapsto \text{ the number of $G$-orbits in } \Omega^{(n)},\\
& F^{\ast}_{G, \Omega}: \mathbb{N} \to \mathbb{N} \cup \{ \infty\}, \quad n \mapsto \text{ the number of $G$-orbits in } \Omega^n,
\end{align*}
and call them \textit{growth rates} of the actions of $G$ on these sets. It is well known that they are nondecreasing, and satisfy the following relations:
\begin{align*}
& f_{G, \Omega} (n) \leqslant F_{G, \Omega} (n) \leqslant n! f_{G, \Omega}(n),\\
& F_{G, \Omega}^{\ast} (n) = \sum_{i=1}^n S(n, i) F_{G, \Omega} (i)
\end{align*}
where $S(n, i)$ is the Stirling number of the second kind. For more details, see \cite{Ca}.

Three cases are of great interest to people, and hence deserve special names.

\begin{definition}
We say that the action of $G$ on $\Omega$ is:
\begin{itemize}
\item \textit{highly transitive} if $F_{G, \Omega} \equiv 1$,

\item \textit{highly homogeneous} if $f_{G, \Omega} \equiv 1$,

\item \textit{oligomorphic} if $f_{G, \Omega} (n)$ is finite for each $n \in \mathbb{N}$.
\end{itemize}
\end{definition}

Clearly, the highly transitive property implies the highly homogeneous property, which in turn implies the oligomorphic property. Furthermore, if the action of $G$ on $\Omega$ is oligomorphic, then the values of $F^{\ast}_{G, \Omega}$ and $F_{G, \Omega}$ on each $n \in \mathbb{N}$ are finite.

Let $H$ be another group acting on $\Omega$. We say that $G$ and $H$ \textit{share the same orbits on $\Omega$} if for each $x \in \Omega$, one has $G \cdot x = H \cdot x$, where $G \cdot x$ is the $G$-orbit containing $x$. The following fact is elementary. For the convenience of the reader, we provide a self-contained proof.

\begin{lemma} \label{orbits of finite cartesian products}
Suppose that $G$ and $H$ act on $\Omega$, and fix $n \geqslant 1$. Then the following are equivalent:
\begin{enumerate}
  \item $G$ and $H$ share the same orbits on $\Omega^n$;
  \item $G$ and $H$ share the same orbits on $\Omega^{(n)}$;
  \item $G$ and $H$ share the same orbits on $\Omega^s$ for every $s \leqslant n$;
  \item $G$ and $H$ share the same orbits on $\Omega^{(s)}$ for every $s \leqslant n$.
\end{enumerate}
\end{lemma}

\begin{proof}
\textbf{(1) $\Rightarrow$ (2).} This is trivial since $\Omega^{(n)}$ is a $G$-subset of $\Omega^n$.

\medskip\noindent
\textbf{(2) $\Rightarrow$ (4).} Fix $s \leqslant n$, and let
\[
\mathbf x = (x_1, \dots, x_s), \qquad \mathbf y = (y_1,\dots, y_s)
\]
be two elements in $\Omega^{(s)}$. Since $\Omega$ is infinite, we can choose an $(n-s)$-tuple $\mathbf z$ whose entries are distinct and disjoint from $\{x_1, \dots, x_s\}$. Clearly, $(\mathbf x, \mathbf z) \in \Omega^{(n)}$.

If $\mathbf x$ and $\mathbf y$ lie in the same $G$-orbit in $\Omega^{(s)}$, we can pick $g \in G$ with $g \cdot \mathbf x = \mathbf y$, and set $\mathbf w = g \cdot \mathbf z$. Then $(\mathbf x, \mathbf z)$ and $(\mathbf y, \mathbf w)$ lie in the same $G$-orbit inside $\Omega^{(n)}$. By (2) they therefore lie in the same $H$-orbit, so some $h \in H$ sends $(\mathbf x,\mathbf z)$ to $(\mathbf y,\mathbf w)$; in particular $h \cdot \mathbf x = \mathbf y$. Thus $\mathbf x$ and $\mathbf y$ lie in the same $H$-orbit. The conclusion of (4) then follows from symmetry.

\medskip\noindent
\textbf{(4) $\Rightarrow$ (3).} Fix a certain $s \leqslant n$ and take an arbitrary
\[
\mathbf x = (x_1,\dots,x_s) \in \Omega^s.
\]
We need to show $G \cdot \mathbf x = H \cdot \mathbf x$; that is, for any $g \in G$, one can find a certain $h \in H$ such that $g \cdot \mathbf x = h \cdot \mathbf x$.

Let $\{x_{i_1},\dots,x_{i_\ell}\}$ be the set of distinct values occurring in $\mathbf x$ (listed in any order) and put
\[
\mathbf z = (x_{i_1},\dots,x_{i_\ell}) \in \Omega^{(\ell)}.
\]
By hypothesis (4), $G$ and $H$ share the same orbits on $\Omega^{(\ell)}$, so $G \cdot \mathbf z = H\cdot \mathbf z$. In particular, one can find a certain $h \in H$ such that $g \cdot \mathbf z = h \cdot \mathbf z$. But this happens if and only if $g \cdot \mathbf x = h \cdot \mathbf x$.

\medskip\noindent
\textbf{(3) $\Rightarrow$ (1).} This is trivial.
\end{proof}

\begin{remark}
The equivalence of (1) and (2) cannot be extended to $\Omega^{\{n\}}$: $\Sym(\mathbb{Q})$ and $\Aut(\mathbb{Q}, \leqslant)$ share the same orbits for each $\mathbb{Q}^{\{n\}}$, but they have distinct orbits on $\mathbb{Q}^2$.
\end{remark}

Recall that the action of $G$ on $\Omega$ induces the \textit{canonical relational structure} $\mathcal{M} = (\Omega, (R_i)_{i \in I})$. Explicitly, each $R_i$ is an equivalence relation in $\Omega^{n_i}$ such that two tuples $\mathbf x$ and $\mathbf y$ in $\Omega^{n_i}$ are contained in $R^{n_i}$ if and only if they lie in the same $G$-orbit. This is a homogeneous structure, meaning that embeddings between finite substructures of $\mathcal{M}$ extend to elements in $G$.\footnote{In literature, a homogeneous structure requires that isomorphisms between finite structures of $\mathcal{M}$ extend to automorphisms of $\mathcal{M}$. But as we shall see later, $G$ is a dense subgroup of $\Aut(\mathcal{M})$, so from the definition of permutation topology we can require these automorphisms to be contained in $G$. Furthermore, if every embedding can be extended, so is every isomorphism. Conversely, if every isomorphism can be extended, then every embedding $f: \Gamma \to \Sigma$ gives an isomorphism $\Gamma \cong f(\Gamma) \subseteq \Sigma$, and the extension of this isomorphism also extends $f$.} Clearly, two groups $G$ and $H$ acting on $\Omega$ induce the same canonical relational structure if and only if they share the same orbits on $\Omega^n$ for all $n \in \mathbb{N}$.

We give two immediate consequences of the above lemma.

\begin{corollary} \label{orbits on cartesian products}
Suppose that $G$ and $H$ act on $\Omega$. Then the following are equivalent:
\begin{enumerate}
\item $G$ and $H$ induce the same relational structure on $\Omega$;

\item for any $N \in \mathbb{N}$, there exists $n \geqslant N$ such that $G$ and $H$ share the same orbits on $\Omega^n$;

\item for any $N \in \mathbb{N}$, there exists $n \geqslant N$ such that $G$ and $H$ share the same orbits on $\Omega^{(n)}$.
\end{enumerate}
\end{corollary}

In some situations we can use the growth rates to determine whether two groups induce the same relational structure.

\begin{corollary} \label{growth functions}
Suppose that the action of $G$ on $\Omega$ is oligomorphic, and let $H$ be a subgroup of $G$. Then the following are equivalent:
\begin{enumerate}
\item $G$ and $H$ induce the same relational structure on $\Omega$;

\item for any $N \in \mathbb{N}$, there exists $n \geqslant N$ such that $F_{G, \Omega} (n) = F_{H, \Omega} (n)$;

\item for any $N \in \mathbb{N}$, there exists $n \geqslant N$ such that $F^{\ast}_{G, \Omega} (n) = F^{\ast}_{H, \Omega} (n)$.
\end{enumerate}
\end{corollary}

\begin{proof}
Clearly, (1) implies both (2) and (3). We show that (2) implies (1), and a similar argument gives the implication $(3) \Rightarrow (1)$.

Given $s \in \mathbb{N}$, one can find a certain $n \geqslant s$ such that $F_{G, \Omega} (n) = F_{H, \Omega} (n)$. Since the action of $G$ on $\Omega$ is oligomorphic, $F_{G, \Omega}(n)$ is finite. Since each $G$-orbit is a disjoint union of $H$-orbits, the equality forces that $G$ and $H$ share the same orbits on $\Omega^{(n)}$, and hence the same orbits on $\Omega^{(s)}$. Since $s$ is arbitrary, we conclude that $G$ and $H$ induce the same relational structure on $\Omega$.
\end{proof}

Now we introduce a topological structure on $G$. Equip $\Omega$ with the discrete topology. Given a subset $\Gamma \subseteq \Omega$, we set
\[
G_{\Gamma} = \Stab_G(\Gamma) = \{ g \in G \mid g \cdot x = x, \, \forall x \in \Gamma \}.
\]
In particular, if $\Gamma = \{ x \}$, we also write $G_x$ for $\Stab_G( \{x\})$. A subset $U \subseteq G$ is open in $G$ if for any $g \in U$, there exists a certain finite set $\Gamma \subseteq \Omega$ such that $g G_{\Gamma} \subseteq U$. If we regard elements in $G$ as maps from $\Omega$ to itself, then this topology on $G$ is precisely the \textit{compact-open topology}, or pointwise convergence topology. With respect to this topology $G$ becomes a topological group, and the stabilizer subgroups of finite subsets of $\Omega$ form a neighborhood basis of the identity. In particular, if $\rho: G \to \Sym(\Omega)$ is injective, then $G$ is totally disconnected. For details, see \cite{Ca}.

The topological structure on $G$ gives us more statements equivalent to those in Corollaries \ref{orbits on cartesian products} and \ref{growth functions}.

\begin{proposition} \label{topological characterization}
Let $H$ be a subgroup of $G$. Then the following are equivalent:
\begin{enumerate}
\item $G$ and $H$ induce the same relational structure on $\Omega$;

\item $H$ is a dense subgroup of $G$;

\item For every $n \in \mathbb{N}$ and every tuple $\mathbf x \in \Omega^n$, the natural map of $H$-sets
    \[
      H/(H \cap G_{\mathbf x}) \cong HG_{\mathbf x}/G_{\mathbf x} \longrightarrow G/G_{\mathbf x}
    \]
induced by inclusion $H \hookrightarrow G$ is surjective.
\end{enumerate}
\end{proposition}

\begin{proof}
We prove the cycle of implications $(2)\Rightarrow(1)\Rightarrow(3)\Rightarrow(2)$.

\medskip

\noindent\textbf{(2) $\Rightarrow$ (1).}
Fix $n \geqslant 1$, a tuple $\mathbf x = (x_1, \dots, x_n) \in \Omega^n$, and $g \in G$. We need to find a certain $h \in H$ such that $g \cdot \mathbf x = h \cdot \mathbf x$. Consider $\Gamma = \{a_1,\dots,a_n\} \subseteq \Omega$. Since $H$ is dense in $G$, $H \cap g G_{\Gamma}$ is nonempty. It is easy to see that any element in this intersection serves as the desired $h$.

\medskip

\noindent\textbf{(1) $\Rightarrow$ (3).}
Identify the $G$-set $G/G_{\mathbf x}$ with $G \cdot \mathbf x$ via $g G_{\mathbf x} \mapsto g \cdot \mathbf x$, and similarly identify the $H$-set $H/ (H \cap G_{\mathbf x})$ with $H \cdot \mathbf x$. The inclusion map $H \hookrightarrow G$ induces a natural $H$-equivariant map
\[
H / (H \cap G_{\mathbf x}) \longrightarrow G/G_{\mathbf x},
\]
which under the identifications above is exactly the inclusion $H \cdot \mathbf x \hookrightarrow G \cdot \mathbf x$. Since $H \cdot \mathbf x = G \cdot \mathbf x$, the induced map is surjective.

\medskip

\noindent\textbf{(3) $\Rightarrow$ (2).}
It suffices to show that $g G_{\Gamma} \cap H \neq \emptyset$ for every finite subset $\Gamma$ of $\Omega$ and every $g \in G$. Enumerate $\Gamma$ as a tuple $\mathbf x  \in \Omega^n$ and note that $G_{\Gamma} = G_{\mathbf x}$. Since the map
\[
H/ (H \cap G_{\mathbf x}) \longrightarrow G / G_{\mathbf x}
\]
is surjective, the coset $g G_{\mathbf x} \in G/G_{\mathbf x}$ has a preimage $h (H \cap G_{\mathbf x})$ with $h \in H$. But $h (H \cap G_{\mathbf x})$ is sent to $h G_{\mathbf x}$ via this map, so $h G_{\mathbf x} = g G_{\mathbf x}$, and hence $h \in gG_{\mathbf x} \cap H$. Thus the intersection is indeed nonempty.
\end{proof}

\begin{remark}
In the situation that $G$ is subgroup of $\Sym(\Omega)$, the implication $(1) \Rightarrow (2)$ follows from the well known fact that $G$ is a dense subgroup of $\Aut(\mathcal{M})$, the automorphism group of the canonical relational structure. Indeed, by (1) and this fact, both $H$ and $G$ are dense subgroups of $\Aut(\mathcal{M})$, so $H$ is also dense in $G$. For details, please refer to \cite[Section 2.4]{Ca}.
\end{remark}

\section{Restriction and equivalences}

Throughout this section let $R$ be a commutative ring equipped with the discrete topology, and suppose that elements of $G$ act on $R$ continuously as ring automorphisms. Let $RG$ be the skew group ring whose multiplication is determined by
\[
(ag) \cdot (bh) = (a g(b)) gh
\]
for $a, b \in R$ and $g, h \in G$.

\begin{example}
Let $k$ be a commutative ring, $\Phi$ a $G$-subset of $\Omega$, and $R = k[\Phi]$ the polynomial ring whose indeterminates are elements in $\Phi$. Then the natural action of $G$ on $R$ is continuous. Indeed, for a polynomial $P \in R$ with indeterminates $x_1, \ldots, x_n$, one has
\[
\Stab_G(P) \supseteq (G_{x_1} \cap \ldots \cap G_{x_n}),
\]
a finite intersection of open subgroups, so $\Stab_G(P)$ is also an open subgroup of $G$.
\end{example}

\begin{definition}
An $RG$-module $V$ is \textit{discrete} if $V$ is equipped with the discrete topology and the action of $G$ on it is continuous, or equivalently, for any $v \in V$, the stabilizer subgroup
\[
G_v = \Stab_G(v) = \{ g \in G \mid g \cdot v = v \}
\]
is an open subgroup of $G$.
\end{definition}

Denote the category of all discrete $RG$-modules by $RG \Mod^{\dis}$. It is a Grothendieck category with enough injectives. Moreover, the family
\[
\{ R(G/H) \mid H \leqslant G \text{ is open} \}
\]
is a set of generators of this category, where $R(G/H)$ is the free $R$-module with a basis consisting of left cosets $gH \in G/H$.

Given a subgroup $H$ of $G$ and a discrete $RG$-module $V$, we can view $V$ as an $RH$-module, and in this circumstance it is a discrete $kH$-module since
\[
\Stab_H(v) = \Stab_G(v) \cap H, \quad \forall v \in V
\]
is an open subgroup of $H$. Therefore, the classical restriction functor from the category of all $RG$-modules to the category of all $RH$-modules induces a restriction functor
\[
\Res: RG \Mod^{\dis} \longrightarrow RH \Mod^{\dis}.
\]

The main goal of this section is to show that the above restriction functor induces an equivalence if and only if $H$ is a dense subgroup of $G$, or equivalently, if and only if $G$ and $H$ induce the same relational structure on $\Omega$.

\begin{lemma} \label{fully faithful}
Let $H$ be a dense subgroup of $G$. Then the restriction functor
\[
\Res: RG \Mod^{\dis} \longrightarrow RH \Mod^{\dis}
\]
is fully faithful.
\end{lemma}

\begin{proof}
The restriction functor is clearly faithful, so we only need to show that it is a full functor. Explicitly, given $V, W \in RG \Mod^{\dis}$, we want to show that the natural map
\[
\Hom_{RG}(V, W) \longrightarrow \Hom_{RH}(\Res V, \Res W)
\]
is surjective. We prove this by checking that any $H$-equivariant map $f : \Res V \to \Res W$ is also $G$-equivariant.

Fix $v \in V$, and define two maps
\begin{align*}
& \varphi: G \to W, \quad g \mapsto f(g \cdot v);\\
& \psi: G \to W, \quad g \mapsto g \cdot f(v).
\end{align*}
Since $G$ acts on $V$ and $W$ continuously, the orbit maps $g \mapsto g \cdot v$ and $g \mapsto g \cdot f(v)$ are continuous. The continuity of $f$ is trivial since both $V$ and $W$ are equipped with the discrete topology. Thus $\varphi$ and $\psi$ are continuous maps.

Since $f$ is $H$-equivariant, we have $\varphi(h) = \psi(h)$ for all $h \in H$. But $H$ is dense in $G$ and $W$ is discrete; the continuity of $\varphi$ and $\psi$ implies that $\varphi = \psi$ on $G$. Indeed, for any $g \in G$, $\{\varphi(g) \}$ and $\{ \psi(g) \}$ are open sets in $W$, so
\[
U = \varphi^{-1}(\{\varphi(g)\}) \cap \psi^{-1}(\{\psi(g)\})
\]
is an open neighborhood of $g$ in $G$. By density of $H$, we can find a certain $h \in H \cap U$, so
\[
\varphi(g) = \varphi(h) = \psi(h) = \psi(g).
\]
But this means that $f(g \cdot v) = g \cdot f(v)$ for all $g \in G$; that is, $f$ is $G$-equivariant at $v \in V$. But $v$ is arbitrarily chosen in $V$, so $f$ is $G$-equivariant everywhere.
\end{proof}

The following lemma states a standard fact of topological group theory.

\begin{lemma} \label{intersection of open subgroups}
Let $T$ be a topological group, and let $T'$ be a dense subgroup of $T$. Then every open subgroup $L \leqslant T'$ is of the form
\[
L = T' \cap K
\]
for some open subgroup $K \leqslant T$.
\end{lemma}

\begin{proof}
Since $L$ is an open subgroup of $T'$, there exists an open subset $W \subseteq T$ such that
\[
L = T' \cap W.
\]
Let $K$ be the closure of $L$ in $T$, which is also a subgroup of $T$. We claim that $K$ is open. Indeed, we have
\[
K = \overline{L} = \overline{W \cap T'} = \overline{W},
\]
where the last equality holds as $T'$ is dense in $T$ and $W$ is open in $T$. But $\overline{W}$ contains the open set $W$, so $K$ is an open subgroup of $T$. Consequently,
\[
T' \cap K = T' \cap \overline{L} = \overline{L}^{\,T'} = L,
\]
where $\overline{L}^{\,T'}$ denotes the closure of $L$ inside $T'$, which equals $L$ since $L$ as an open subgroup of $T'$ is also closed.
\end{proof}

\begin{proposition} \label{essential surjectivity}
Let $H$ be a dense subgroup of $G$. Then the restriction functor
\[
\Res: RG \Mod^{\dis} \longrightarrow RH \Mod^{\dis}
\]
is essentially surjective.
\end{proposition}

\begin{proof}
We prove the conclusion by showing that for any discrete $RH$-module $V$, one can extend the continuous action of $H$ on $V$ uniquely to a continuous action of $G$ on $V$.

Take $v \in V$ and $g \in G$. Note that the stabilizer subgroup $H_v = \{ h \in H \mid h \cdot v = v\}$ is open in $H$, so by Lemma \ref{intersection of open subgroups}, one can find an open subgroup $K \leqslant G$ with $H_v = K \cap H$. Note that here we cannot take $K$ to be $G_v$ as we have not shown that $G_v$ is open.

Since $H$ is dense in $G$, there exists a net $(h_i)_{i\in I}$ in $H$ with $h_i \to g$. We claim that the net $(h_i \cdot v)$ in $V$ is eventually constant. Indeed, because the net $(h_i)$ is Cauchy, for the neighbourhood $K$ of the identity in $G$, one can find an index $i_0$ such that
\[
h_i^{-1}h_j \in K, \quad \forall \, i,j \geqslant i_0,
\]
and hence
\[
h_i^{-1} h_j \in K \cap H = H_v, \quad \forall \, i,j \leqslant i_0.
\]
Consequently, in this case
\[
h_i \cdot v = (h_i h_j^{-1}) \cdot(h_j \cdot v) = h_j \cdot v,
\]
so $(h_i \cdot v)$ is eventually constant as claimed. Thus we can define
\[
g \cdot v = \lim_{i \in I} (h_i \cdot v).
\]

This definition is independent of the choice of nets. Indeed, if $(h_i)$ and $(h'_j)$ are two nets in $H$ both converging to $g$, again by the Cauchy property applied to the combined net, there exist indices $i_0, j_0$ such that for all $i \geqslant i_0$ and $j \geqslant j_0$ we have $h_i^{-1} h'_j \in K$. Thus $h_i^{-1} h'_j \in K \cap H = H_v$, so $h'_j \cdot v = h_i \cdot v$. Therefore, the eventual values of the two nets coincide.

It is easy to see that the above definition respects the group action axioms. We check that it is continuous. Since $H$ is dense in $G$ and $K$ is open in $G$, $K \cap H$ is dense in $K$. Thus for any $x \in K$, one can find a net $(x_i)_{i \in I} \subseteq K \cap H$ with $x_i \to x$. By our construction, $x \cdot v$ is the limit of the net $(x_i \cdot v)$. But $x_i \cdot v = v$ for all $i$, so $x \cdot v = v$ as well. Thus $x \in G_v$, and hence $K \leqslant G_v$. Consequently, $G_v$ is open in $G$, so the action is indeed continuous.

We check the above definition also respects the module structure over the skew group ring $RG$. Take $r \in R$, $g \in G$, $v \in V$, and choose a net $(h_i)$ in $H$ with $h_i \to g$. Because the action of $G$ on the $R$ is continuous, the net $(h_i \cdot r)$ is eventually constant and equals $g \cdot r$. Likewise $(h_i \cdot v)$ is eventually constant and equals $g \cdot v$. Thus for large enough $i$, one has
\[
h_i \cdot (r v) = (h_i \cdot r) \,(h_i \cdot v) = (g \cdot r) \, (g\cdot v),
\]
so by taking limit one gets $g \cdot (r v) = (g\cdot r)\,(g\cdot v)$ as desired.

Finally, we mention that the above extension of the $RH$-module structure on $V$ to the $RG$-module structure on $V$ is unique. Indeed, if $\phi: G \times V$ and $\psi: G \times V \to V$ are two extensions of the given $H$-action on $V$, they must be the same since they coincide in the dense subgroup $H$ of $G$.

\end{proof}

Now we are ready to prove the main result.

\begin{theorem} \label{res is an equivalence}
Suppose that both $G$ and $H$ act on $\Omega$. If they induce the same relational structure $\mathcal{M}$ on $\Omega$, then one has an equivalence
\[
RG \Mod^{\dis} \simeq RH \Mod^{\dis}.
\]
If furthermore $H$ is a subgroup of $G$, then this equivalence can be induced by the restriction functor.
\end{theorem}

\begin{proof}
Let $\rho: G \to \Sym(\Omega)$ be the group homomorphism determined by the action of $G$ on $\Omega$, and set $\overline{G} = G / \ker{\rho}$. Define $\overline{H}$ similarly. It is easy to check that $G$, $\overline{G}$, $H$ and $\overline{H}$ induce the same relational structure on $\Omega$, and
\[
RG \Mod^{\dis} \simeq R\overline{G} \Mod^{\dis}, \quad RH \Mod^{\dis} \simeq R\overline{H} \Mod^{\dis}.
\]
Furthermore, by \cite[Section 2.4]{Ca}, the closure of $\overline{G}$ and $\overline{H}$ in $\Sym(\Omega)$ is $\widetilde{G} = \Aut(\mathcal{M})$. By Lemma \ref{fully faithful} and Proposition \ref{essential surjectivity}, the restriction functors induce equivalences
\[
R\widetilde{G} \Mod^{\dis} \simeq R\overline{G} \Mod^{\dis}, \quad R\widetilde{G} \Mod^{\dis} \simeq R\overline{H} \Mod^{\dis},
\]
establishing the first statement. The second statement immediately follows from previous results.
\end{proof}

\begin{remark}
Strictly speaking, the restriction functor only gives one half of the equivalence, and one may wonder to complete the picture by finding the functor of the other direction. The proper candidate is the \textit{coinduction functor} which is the right adjoint of $\Res$; see \cite{Hr}.\footnote{In that paper, deviating from the convention, the author call the right adjoint \textit{induction functor} and the left adjoint \textit{coinduction functor}.}
\end{remark}

The only if direction of the second statement of this theorem also holds. In the rest of this section we give a clear statement and prove it.

\begin{lemma} \label{density}
A subgroup $H \leqslant G$ is dense in $G$ if and only if $HK = G$ for every open subgroup $K \leqslant G$.
\end{lemma}

\begin{proof}
The subgroup $H$ is dense in $G$ if and only if for every $g \in G$ and every open subgroup $K \leqslant G$, the intersection $gK \cap H$ is not empty, and if and only if there exists a certain $x \in K$ such that $gx \in H$. But this is equivalent to $g \in Hx^{-1} \subseteq HK$.
\end{proof}

The above conclusion actually holds for any topological group with a neighborhood basis of the identity consisting of open subgroups.

\begin{proposition} \label{the other direction}
Let $H$ be a subgroup of $G$. If the restriction functor induces an equivalence
\[
\Res: RG \Mod^{\dis} \to RH \Mod^{\dis},
\]
then $H$ is dense in $G$.
\end{proposition}

\begin{proof}
Suppose that $H$ is not dense in $G$. By Lemma \ref{density}, there exists an open subgroup $K \leqslant G$ with $HK \neq G$. Let $V = R(G/K)$, the free $R$-module with basis $G/K$, endowed with the $G$-action
\[
g \cdot (axK \big) = (g \cdot a)(gxK), \quad x,g \in G, \, a \in R.
\]
It is easy to check that $V$ is a discrete $RG$-module.

Define $f: V \to R$ by
\[
f(agK)=
\begin{cases}
a & \text{if } gK \in HK,\\
0 & \text{if }gK \notin HK,
\end{cases}
\]
for $a \in R$ and $g \in G$ and extend $R$-linearly, where $HK$ means the $H$-orbit in $G/K$ generated by $K \in G/K$. It is easy to check that $f$ is $H$-equivariant.

On the other hand, because $HK\neq G$, we can find $g_0K \in HK$ and $g_1K \notin HK$. Let $g = g_1g_0^{-1}$. Then
\[
f\big(g\cdot(g_0K)\big)=f(g_1K)=0 \neq 1 = g \cdot 1 = g\cdot f(g_0K)
\]
so $f$ is not $G$-equivariant. Consequently, the restriction map
\[
\Hom_{RG} (V, R)\longrightarrow \Hom_{RH} (V, R)
\]
is not surjective, so the restriction functor is not full, and hence cannot be an equivalence.
\end{proof}

\begin{example}
Let $\tilde{\Omega}$ be the Cartesian product of several copies of $\Omega$, and let $G = \Sym(\Omega)$ acting on $\tilde{\Omega}$ diagonally. This action is not highly homogeneous, but oligomorphic. Let $k$ a commutative ring, and $R = k[\tilde{\Omega}]$. In \cite[Question 1.7]{LPY}, the authors ask for a classification of subgroups $H \leqslant G$ such that the restriction functor gives an equivalence $RG \Mod^{\dis} \simeq RH \Mod^{\dis}$. Now the answer is clear: this happens if and only if the action of $H$ on $\Omega$ is highly transitive. Indeed, this equivalence holds if and only if $H$ is a dense subgroup of $G$, viewed as a subgroup of the topological group $\Sym(\tilde{\Omega})$. But according to Proposition \ref{topological characterization} and Corollary \ref{orbits on cartesian products}, $H$ is a dense subgroup of $G$ if and only if they share the same orbits on $\tilde{\Omega}^n$ for infinitely many $n$, and if and only if $G$ and $H$ share the same orbits for each $\Omega^n$; that is, the action of $H$ on $\Omega$ is highly transitive.
\end{example}

\section{Orbit categories}

In this section we consider two categories associated to the action of $G$ on $\Omega$: the orbit category and the category of finite substructures and embeddings. We show that under a mild condition they are actually isomorphic.

The orbit category $\mathcal{G}$ of $G$ with respect to the neighborhood basis of the identity
\[
\mathcal{B} = \{ \Stab_G(\Gamma) \mid G \subseteq \Omega, \, |\Gamma| < \infty \}
\]
is defined as follows: its objects are left cosets $G/H$ with $H \in \mathcal{B}$, and morphisms from $G/H$ to $G/K$ are $G$-equivariant maps. It is well known that these maps are induced by elements $g \in G$. Explicitly, every $G$-equivariant map $G/H \to G/K$ is of the form $\sigma_g: xH \mapsto xg^{-1}K$ where $g \in G$ satisfies $gHg^{-1} \leqslant K$. The map $g \mapsto \sigma_g$ defines a surjective map
\[
N_G(H, K) = \{g \in G \mid gHg^{-1} \subseteq K \} \longrightarrow \mathcal{G}(G/H, G/K),
\]
and $\sigma_g = \sigma_h$ if and only if $Kg = Kh$.

Orbit categories play a fundamental role in the study of discrete representations of $G$. Explicitly, we can impose the atomic Grothendieck topology $J_{at}$ on $\mathcal{G}$. By \cite[Proposition 2.1 and Theorem 2.2]{LPY}, there is a unique structure sheaf $\mathcal{R}$ on the atomic site $(\mathcal{G}, J_{at})$ such that
\[
RG \Mod^{\dis} \simeq \Sh(\mathcal{A}^{\op}, J_{at}, \mathcal{R}),
\]
the category of sheaves of modules over the ringed sites; see \cite{Jo, LPY, MM}. In particular, this equivalence tells us that the following result is essentially equivalent to the second statement of Theorem \ref{res is an equivalence}.

\begin{proposition}
Suppose that $H$ is a dense subgroup of $G$. Then their orbit categories $\mathcal{G}$ and $\mathcal{H}$ are isomorphic, and the structure sheaves $\mathcal{R}_G$ over $(\mathcal{G}, J_{at})$ and $\mathcal{R}_H$ over $(\mathcal{H}, J_{at})$ are also isomorphic when we identify these two sites.
\end{proposition}

\begin{proof}
Construct a functor $\mu: \mathcal{G} \to \mathcal{H}$ as follows:
\begin{itemize}
\item \textbf{On objects.} For a finite $\Gamma \subseteq \Omega$, we set $\mu(G/G_{\Gamma}) = H/H_{\Gamma}$, where $H_{\Gamma} = H \cap G_{\Gamma}$. It is easy to check this map is well-defined and surjective. It is also injective. Indeed, since $H_{\Gamma}$ is dense in $G_{\Gamma}$ and $G_{\Gamma}$ is closed in $G$, the closure of $H_{\Gamma}$ in $G$ is precisely $G_{\Gamma}$. Thus if $G_{\Gamma} \neq G_{\Sigma}$, then $H_{\Gamma} \neq H_{\Sigma}$.

\item \textbf{On morphisms.} Given $f \in \Hom_G(G/G_{\Gamma}, G/G_{\Sigma})$, we can find an element $g\in G$ with $g^{-1}G_{\Gamma} g \leqslant G_{\Sigma}$. Choose $h \in H$ with $hG_{\Sigma} = gG_{\Sigma}$, which is possible by Lemma \ref{density}, and define $\mu(f)$ to be the $H$-equivariant map sending $H_{\Gamma}$ to $hH_{\Sigma}$.
\end{itemize}

It is tedious but routine to check that $\mu$ is indeed a functor. Since it is a bijection on object sets, we only need to show that it is fully faithful. This follows from the identifications
\begin{align*}
\Hom_H (H/H_{\Gamma}, H/H_{\Sigma}) & = \{ h \in H \mid h H_{\Gamma} h^{-1} \leqslant H_{\Sigma} \} / H_{\Sigma}\\
& = \{ h \in H \mid hG_{\Gamma} h^{-1} \leqslant G_{\Sigma} \} / H_{\Sigma}\\
& = \{ g \in G \mid gG_{\Gamma} g^{-1} \leqslant G_{\Sigma} \} / G_{\Sigma}\\
& = \Hom_G (G/G_{\Gamma}, G/G_{\Sigma}).
\end{align*}
To prove the second identification, we just observe that the closure of $H_{\Gamma}$ in $G$ is precisely $G_{\Gamma}$, so if the conjugate map by $h$ sends $H_{\Gamma}$ into $H_{\Sigma}$, then it must send the closure $G_{\Gamma}$ into $G_{\Sigma}$. To prove the third identifications, we use Lemma \ref{density} to write every $g \in G$ as $g = h g'$ with $g' \in G_{\Gamma}$ and $h \in H$.

We have constructed an explicit isomorphism $\mu: \mathcal{G} \to \mathcal{H}$ to identify $(\mathcal{G}, J_{at})$ and $(\mathcal{H}, J_{at})$. By \cite[Section 2]{LPY}, for an object $G/G_{\Gamma}$ in $\mathcal{G}$, one has
\[
\mathcal{R}_G (G/G_{\Gamma}) = R^{G_{\Gamma}} = \{ r \in R \mid g \cdot r = r, \, \forall g \in G_{\Gamma} \},
\]
which is the same as
\[
\mathcal{R}_H (H/H_{\Gamma}) = R^{H_{\Gamma}} = \{ r \in R \mid h \cdot r = r, \, \forall h \in H_{\Gamma} \}
\]
since if $r \in R$ is fixed by the dense subgroup $H_{\Gamma}$, then it must be fixed by its closure $G_{\Gamma}$.
\end{proof}

The orbit category $\mathcal{G}$ is intrinsic as it is completely determined by the structure of $G$. We can also use the action of $G$ on $\Omega$ to construct an extrinsic category. Let $\mathcal{M} = (\Omega, (R_i)_{i \in I})$ be the canonical relational structure on $\Omega$ induced by the action of $G$. Let $\mathcal{A}$ be the category of finite substructures of $\mathcal{M}$ and embeddings, and $\Age(\mathcal{M})$ the object set of this category. By abuse of notation, we often denote a finite substructure by its underlying set.

\begin{definition}
We say that $\Age(\mathcal{M})$ has the amalgamation property if given finite substructures $\Sigma, \Gamma_1, \Gamma_2$ of $\mathcal{M}$ as well as embeddings $f_i: \Sigma \to \Gamma_i$, there exist a finite substructure $\Delta$ and embeddings $g_i: \Gamma_i \to \Delta$ such that the following diagram commutes:
\[
\xymatrix{
\Sigma \ar[r]^-{f_1} \ar[d]_-{f_2} & \Gamma_1 \ar[d]^-{g_1}\\
\Gamma_2 \ar[r]_-{g_2} & \Delta.
}
\]
It satisfies the strong amalgamation property if in the above diagram, one can choose $g_1$ and $g_2$ such that for $b_1 \in \Gamma_1$ and $b_2 \in \Gamma_2$, whenever $g_1(b_1) = g_2(b_2)$, there exists some $a \in \Sigma$ such that $b_1 = f_1(a)$ and $b_2 = f_2(a)$.
\end{definition}

It is easy to see that the amalgamation property is precisely the right Ore condition on $\mathcal{A}$, and the strong amalgamation property is equivalent to saying that $\Delta$ in the above diagram is a pushout.

\begin{lemma} \label{strong amalgamation}
The following are equivalent:
\begin{enumerate}
\item strong amalgamation property holds for $\Age(\mathcal{M})$;

\item for every finite set $\Gamma \subseteq \Omega$, the stabilizer subgroup $G_{\Gamma}$ has no finite orbits outside $\Gamma$;

\item for every finite set $\Gamma \subseteq \Omega$, the stabilizer subgroup $G_{\Gamma}$ has no fixed point outside $\Gamma$.
\end{enumerate}
\end{lemma}

\begin{proof}
For the equivalence between (1) and (2), see \cite[Section 2.7. (2.15)]{Ca}.

$(2) \Rightarrow (3)$: This is trivial.

$(3) \Rightarrow (2)$: Suppose that there is a finite $\Gamma \subseteq \Omega$ such that $G_{\Gamma}$ has a finite orbit $\Sigma \subseteq \Omega \setminus \Gamma$. Pick $s \in \Sigma$, and let $\Sigma' = \Sigma \setminus \{s\}$ and $H = G_{\Gamma \cup \Sigma'}$. Clearly, $H$ is a subgroup of $G_{\Gamma}$, so $\Sigma$ is an $H$-orbit. Since $H$ fixes every element in $\Sigma'$, it also fixes $s$. On the other hand, $\Gamma \cup \Sigma'$ is finite, so by (3), $H$ has no fixed point outside $\Gamma \cup \Sigma'$. This contradiction tells us that $G_{\Gamma}$ cannot have a finite orbit in $\Omega \setminus \Gamma$, and (2) holds.
\end{proof}

Now we give a categorical characterization of the strong amalgamation property.

\begin{theorem} \label{orbit categories}
The strong amalgamation property holds for $\Age(\mathcal{M})$ if and only if the map
\[
\phi: \Ob(\mathcal{A}) \longrightarrow \Ob(\mathcal{G}), \quad \Gamma \mapsto G/G_{\Gamma}
\]
induces an isomorphism from $\mathcal{A}^{\op}$ to $\mathcal{G}$.
\end{theorem}

\begin{proof}
Firstly, we show that $\phi$ is a bijection on object sets if and only if the strong amalgamation property holds for $\Age(\mathcal{M})$. Since $\phi$ is clearly surjective, it suffices to consider its injectivity. But $\phi$ is not injective if and only if there are two distinct finite substructures $\Gamma$ and $\Sigma$ such that $G_{\Gamma} = G_{\Sigma}$. Without loss of generality suppose that there is a certain $x \in \Gamma \setminus \Sigma$. Then $x \notin \Sigma$ and is fixed by $G_{\Sigma}$. But by Lemma \ref{strong amalgamation}, this happens if and only if $\Age(\mathcal{M})$ does not have the strong amalgamation property.

It remains to construct an explicit isomorphism from $\mathcal{A}^{\op}$ to $\mathcal{G}$, which is still denoted by $\phi$. We only need to describe its definition on morphism sets, and check that it is fully faithful.

A morphism $\varsigma: \Sigma \to \Gamma$ in $\mathcal{A}^{\op}$ corresponds to a unique embedding $\sigma: \Gamma \to \Sigma$. By the homogeneity of $\Age(\mathcal{M})$, $\sigma$ extends to an element $\tilde{\sigma} \in G$. Since $\tilde{\sigma}^{-1}G_{\Sigma} \tilde{\sigma}$ is contained in $G_{\Gamma}$, by the definition of the orbit category, $\tilde{\sigma}^{-1}$ represents a morphism
\[
G/G_{\Sigma} \longrightarrow G/G_{\Gamma}, \quad gG_{\Sigma} \mapsto g \tilde{\sigma} G_{\Gamma}
\]
in $\mathcal{G}$. Set $\phi(\varsigma)$ to be this morphism.

This construction is well defined. Indeed, if $\tilde{\sigma}$ and $\tilde{\tau}$ are extensions of $\sigma: \Gamma \to \Sigma$ and $\tau: \Gamma \to \Sigma$ in $\mathcal{A}$ respectively, then for any $g \in G$, one has
\begin{align*}
& g \tilde{\sigma} G_{\Gamma} = g \tilde{\tau} G_{\Gamma} \quad \Longleftrightarrow \quad \tilde{\sigma} G_{\Gamma} = \tilde{\tau} G_{\Gamma} \quad \Longleftrightarrow \quad \tilde{\tau}^{-1} \tilde{\sigma} G_{\Gamma} = G_{\Gamma}\\
& \Longleftrightarrow \quad \tilde{\tau}^{-1} \tilde{\sigma} \in G_{\Gamma} \quad \Longleftrightarrow \quad \tilde{\tau}^{-1} \tilde{\sigma}(x) = x, \, \forall x \in \Gamma\\
& \Longleftrightarrow \tilde{\tau}^{-1} \sigma (x) = x, \, \forall x \in \Gamma \quad \Longleftrightarrow \quad \sigma(x) = \tilde{\tau} (x) = \tau(x), \, \forall x \in \Gamma.
\end{align*}
Consequently, if $\tilde{\sigma}$ and $\tilde{\tau}$ are different extensions of $\sigma$, then they represent the same morphism in $\mathcal{G}$, so this construction is independent of the choice of extensions, and hence well defined. The above argument also shows that $\phi$ is faithful. The reader can easily check the functoriality of $\phi$, and it is full since each morphism in $\mathcal{A}$ extends to an element in $G$ and morphisms in $\mathcal{G}$ are represented by elements in $G$.
\end{proof}

\begin{example}
Let $G = \Sym(\Omega)$ and $\tilde{\Omega}$ the Cartesian product of $n$ copies of $\Omega$, and let $\mathcal{M}$ be the relational structure induced by the action of $G$ on $\tilde{\Omega}$. Then $\Age(\mathcal{M})$ satisfies the strong amalgamation property if and only if $n = 1$. Indeed, for $n \geqslant 2$, we can take distinct elements $x_1, \ldots, x_n \in \Omega$, and let $H$ be the stabilizer subgroup of $\mathbf x = (x_1, \ldots, x_n) \in \tilde{\Omega}$. Then $H$ also fixes $(x_1, \ldots, x_1) \in \tilde{\Omega} \setminus \{ \mathbf x \}$. Consequently, $\mathcal{A}$ is not isomorphic to $\mathcal{G}^{\op}$ for $n \geqslant 2$. Indeed, $\mathcal{A}$ varies for different $n$. However, since stabilizer subgroups of finite subsets of $\tilde{\Omega}$ coincide with stabilizer subgroups of finite subsets of $\Omega$, it follows that $\mathcal{G}^{\op}$ is independent of $n$, which is always isomorphic to the category of finite subsets of $\Omega$ and injections.
\end{example}

The \textit{self-embedding functor} and \textit{shift functor} induced by it, play a central role in the representation theory of the category of finite sets and injections and other related categories; see for instances \cite{CEFN, GL, GS, Nagpal}. In the rest of this section we give a characterization of the existence of this kind of embedding functors in terms of homogeneity of group actions.

Let $x$ be a fixed element in $\Omega$, $\Omega_x = \Omega \setminus \{x\}$, and $G_x = \Stab_G(x)$. Let $\Age(\mathcal{M})_x$ be the set of finite substructures $\Gamma$ of $\mathcal{M}$ such that $x \notin \Gamma$, and $\mathcal{A}_x$ the full subcategory of $\mathcal{A}$ with object set $\Age(\mathcal{M})_x$. There is a natural map
\[
\iota: \Age(\mathcal{M})_x \longrightarrow \Age(\mathcal{M}), \quad \Gamma \mapsto \Gamma \sqcup \{x\}.
\]
Thus one may ask under what conditions it induces an embedding functor $\iota: \mathcal{A}_x \to \mathcal{A}$ such that for $f: \Gamma \to \Sigma$ in $\mathcal{A}_x$, $\iota(f)$ is the morphism
\[
\Gamma \sqcup \{x\} \longrightarrow \Sigma \sqcup \{x\}
\]
extending $f$ and fixing $x$.

\begin{proposition}
The map $\iota$ induces such an embedding functor if and only if $\Age(\mathcal{M})_x$ is homogeneous with respect to the action of $G_x$; that is, every embedding between two members in $\Age(\mathcal{M})_x$ extends to an element in $G_x$.
\end{proposition}

\begin{proof}
\textbf{The if direction.} Suppose that $\Age(\mathcal{M})_x$ is homogeneous with respect to $G_x$. Given a morphism $f:\Gamma \to \Sigma$ in $\mathcal{A}_x$, by homogeneity, there exists $\tilde f \in G_x$ extending $f$, so we can define $\iota(f)$ to be its restriction to $\Gamma \sqcup \{x\}$. Then $\iota(f)$ is an embedding $\Gamma \sqcup \{x\} \to \Sigma \sqcup \{x\}$ with $\iota(f)(x) = x$. Clearly, this construction is functorial, so $\iota$ is a functor.

\textbf{The only if direction.} Suppose that $\iota$ defines an embedding functor. Given an embedding $f:\Gamma \to \Sigma$ in $\mathcal{A}_x$, we obtain an embedding
\[
\iota(f): \Gamma \sqcup \{x\} \to \Sigma \sqcup \{x\}
\]
in $\mathcal{A}$ fixing $x$. By the homogeneity of $\mathcal{M}$, $\iota(f)$ extends to some $\tilde f \in G$. Since $\iota(f)(x) = x$, it follows that $\tilde f \in G_x$. Thus $f$ extends to an element of $G_x$, showing that $\Age(\mathcal{M})_x$ is homogeneous with respect to $G_x$.
\end{proof}

\begin{remark}
The action of $G_x$ on $\Omega_x$ also induces a canonical relational structure $\mathcal{M}_x$ on $\Omega_x$. Using standard arguments in Fr\"{a}iss\'{e} theory, the reader can show the equivalence of the following statements:
\begin{enumerate}
\item the map $\iota$ induces such an embedding functor from $\mathcal{A}_x$ to $\mathcal{A}$;

\item $\Age(\mathcal{M})_x$ is homogeneous with respect to the action of $G_x$;

\item $\Age(\mathcal{M}_x) = \Age(\mathcal{M})_x$;

\item $G_x$ is a dense subgroup of $\Aut(\mathcal{M}_x)$.
\end{enumerate}
\end{remark}

\begin{example}
We use a few examples to illustrate the above result.
\begin{enumerate}
\item Let $\Omega = \mathbb{N}$, $G = \Sym(\mathbb{N})$, and $x = 0$. Then $G_x = \Sym(\mathbb{Z_+})$ acts highly transitively on $\Omega_x = \mathbb{Z}_+$. In this case the embedding functor $\iota$ is precisely the one constructed in \cite{CEFN}.

\item Let $\Omega = \mathbb{Q}$, $G = \Aut(\mathbb{Q}, \leqslant)$, and $x = 0$. The action of $G_x$ on $\Omega_x$ is no longer highly homogeneous (actually, it is not even transitive). In this case the map $\iota$ does not induces an embedding functor from $\mathcal{A}_x$ to $\mathcal{A}$ since the isomorphism $\{-1\} \to \{1\}$ cannot extend to an isomorphism in $G_x$.

\item One can construct a self-embedding functor from the category of finite linearly ordered sets and order-preserving injections to itself via adding an extra minimal element; see \cite{GL, GS}. The reader may expect to use a similar construction as well as the above proposition to recover this functor. Unfortunately, this is not the case. Explicitly, let $\Omega = \mathbb{Q} \sqcup \{x\}$ and extend the linear order on $\mathbb{Q}$ to a linear order $\leqslant$ on $\Omega$ such that $x$ is the minimal element. Let $G = \Aut(\Omega, \leqslant)$. Then $G_x = G = \Aut(\mathcal{M}_x)$. Consequently, by the previous proposition we indeed obtain an embedding functor $\iota$ from $\mathcal{A}_x$ to $\mathcal{A}$. But this is completely different from the self-embedding functor in \cite{GL, GS}.

    The reason to cause this difference is subtle. If we consider the canonical relational structure $\mathcal{M}$ on $\Omega$, then the map $\{x \} \to \{r\}$ with $r \in \mathbb{Q}$ is not an isomorphism of finite substructures since $x$ and $r$ lie in distinct $G$-orbits. On the other hand, if we consider the relational structure $\mathcal{M}'$ consisting of the single binary relation, namely the linear order on $\Omega$, then this map is indeed an isomorphism of finite substructures. Consequently, if we define $\mathcal{A}'$ and $\mathcal{A}'_x$ for $\Age(\mathcal{M}')$ and $\Age(\mathcal{M}')_x$ respectively, then the self-embedding functor in \cite{GL, GS} is an embedding functor from $\mathcal{A}'_x$ to $\mathcal{A}'$, and can be obtained via the construction described before the previous proposition. But it is not an embedding functor from $\mathcal{A}_x$ to $\mathcal{A}$ since $\Age(\mathcal{M})$ and $\Age(\mathcal{M}')$ are different, though $\mathcal{M}'$ completely determines $\mathcal{M}$ because $\Aut(\Omega, \mathcal{M}) = \Aut(\Omega, \mathcal{M}')$.
\end{enumerate}
\end{example}

\section{Noetherianity}

Throughout this section without loss of generality let $G$ be a subgroup of $\Sym(\Omega)$, and suppose that the action of $G$ on $\Omega$ is highly homogeneous. Let $k$ a commutative Noetherian ring, and $R = k[\Omega]$. The following classification of highly homogeneous groups was proved by Cameron in \cite{Ca76}.

\begin{theorem}[Theorem 3.10 \cite{Ca}] \label{classification}
Suppose that the action of $G$ on $\Omega$ is highly homogeneous. Then $G$ is a dense subgroup (in the topological sense) of the following automorphism groups:
\begin{enumerate}
\item $\Sym(\Omega)$;

\item $\Aut(\Omega, \leqslant)$ where $\leqslant$ is an unbounded dense linear order on $\Omega$,

\item $\Aut(\Omega, \mathcal{B})$ where $\mathcal{B}$ is a dense (in the sense of order theory) linear betweenness relation,

\item $\Aut(\Omega, \mathcal{K})$ where $\mathcal{K}$ is a dense (in the sense of order theory) cyclic order,

\item $\Aut(\Omega, \mathcal{S})$ where $\mathcal{S}$ is a dense (in the sense of order theory) separation relation.
\end{enumerate}
\end{theorem}

For explicit definitions and examples of these relations, please refer to \cite{BMMN, Ca76, Ca}.

By this result and Theorem \ref{res is an equivalence}, to consider discrete representations of $G$, without loss of generality one may assume that $G$ is one of the above five groups.

\begin{lemma}
The strong amalgamation property holds for $\Age(\mathcal{M})$.
\end{lemma}

\begin{proof}
By \cite[Theorems 11.7, 11.9, and 11.11]{BMMN}, one has
\[
\Aut(\Omega, \leqslant) \leqslant \Aut(\Omega, \mathcal{B}) \leqslant \Sym(\Omega), \quad \Aut(\Omega, \mathcal{K}) \leqslant \Aut(\Omega, \mathcal{S}) \leqslant \Sym(\Omega).
\]
By Lemma \ref{strong amalgamation}, it suffices to check the conclusion for $\Aut(\Omega, \leqslant)$ and $\Aut(\Omega, \mathcal{K})$. For $G = \Aut(\mathbb{Q}, \leqslant)$, this is clear. Indeed, given a finite subset $\Gamma \subseteq \Omega$ and list its elements as $x_1 < x_2 < \ldots < x_n$, then the orbits of $G_{\Gamma}$ on $\Omega \setminus \Gamma$ are precisely open intervals $(-, x_1)$, $(x_1, x_2)$, $\ldots$, and $(x_n, -)$, all of which are infinite sets since the linear order is dense and unbounded.

Now suppose that $G = \Aut(\Omega, \mathcal{K})$. By \cite[Theorem 11.11]{BMMN}, for a fixed $x \in \Omega$ and $\Omega_x = \Omega \setminus \{ x \}$, one has $H = \Aut(\Omega_x, \leqslant) \leqslant G$, where the restricted linear order $\leqslant$ on $\Omega_x$ is also unbounded and dense. Now if $\Gamma$ is a finite subset of $\Omega$ such that $G_{\Gamma}$ has a fixed point $y \in \Omega \setminus \Gamma$, then $y$ is also fixed by $H_{\Gamma} = H \cap G_{\Gamma}$. But we just proved in the previous paragraph (replace $\Omega$ by $\Omega_x$ and $G$ by $H$) that $H_{\Gamma}$ has no fixed point outside $\Gamma$, so $y \in \Gamma$. The conclusion then follows from this contradiction.
\end{proof}

By Theorem \ref{orbit categories}, the category $\mathcal{G}^{\op}$ is isomorphic to $\mathcal{A}$ consisting of finite substructures of $\mathcal{M}$ and embeddings. But in this case two finite substructures are isomorphic if and only if they have the same cardinality,\footnote{Indeed, if the action is highly homogeneous, then for any two finite subsets with the same cardinality, we can find an element $g \in G$ sending one set to another, so they are isomorphic. Conversely, if two finite substructures are isomorphic, they of course have the same cardinality. The reader can see that this `if and only if' statement actually characterizes the highly homogeneous property by noting that every isomorphism between two finite substructures is given by the restriction of an element in $G$.} so $\mathcal{A}$ is equivalent to one of the following categories:
\begin{itemize}
\item $\FI$: the category with objects $[n]$, $n \in \mathbb{N}$, and injections;
\item $\OI$: the category with objects $[n]$ equipped with the canonical linear order, $n \in \mathbb{N}$, and order-preserving injections;
\item $\BI$: the category with objects $[n]$ equipped with the canonical linear betweenness relation, $n \in \mathbb{N}$, and relation-preserving injections;
\item $\CI$: the category with objects $[n]$ equipped with the canonical cyclic order, $n \in \mathbb{N}$, and order-preserving injections;
\item $\SI$: the category with objects $[n]$ equipped with the canonical separation relation, $n \in \mathbb{N}$, and relation-preserving injections.
\end{itemize}

Denote one of the above categories by $\C$. Since $\mathcal{A}$ and $\C$ satisfies the right Ore condition, we can impose the atomic Grothendieck topology $J_{at}$ on them, and call $(\mathcal{A}^{\op}, J_{at})$ and $(\C^{\op}, J_{at})$ the atomic sites. By \cite[Proposition 2.1 and Theorem 2.2]{LPY}, there are structure sheaves $\tilde{\mathcal{R}}$ and $\mathcal{R}$ on these sites such that
\[
RG \Mod^{\dis} \simeq \Sh(\mathcal{A}^{\op}, J_{at}, \tilde{\mathcal{R}}) \simeq \Sh(\C^{\op}, J_{at}, \mathcal{R}),
\]
the category of sheaves of modules over the ringed sites. For details, see \cite{Jo, LPY, MM}.

The following lemma gives an explicit description of $\mathcal{R}$.

\begin{lemma} \label{structure sheaf}
The structure sheaf $\tilde{\mathcal{R}}$ over $(\mathcal{A}{\op}, J_{at})$ has the following description:
\begin{itemize}
\item for each finite subset $\Gamma \subseteq \Omega$, the value $\mathcal{R}(\Gamma)$ of $\mathcal{R}$ on $\Gamma$ is $k[\Gamma]$;

\item for each embedding $\phi: \Gamma \to \Sigma$, the induced ring homomorphism is
\[
\mathcal{R}(\phi): k[\Gamma] \to k[\Sigma], \quad x \mapsto \phi(x).
\]
\end{itemize}
\end{lemma}

\begin{proof}
Take $\Gamma$ be a finite subset of $\Omega$. By \cite[Section 2]{LPY}, $\mathcal{R}(\Gamma)$ is the set of fixed points of $R$ by $G_{\Gamma}$; that is,
\[
\mathcal{R} ([n]) = \{P \in k[\Omega] \mid g \cdot P = P, \, \forall g \in G_{\Gamma} \}.
\]
Clearly, $k[\Gamma] \subseteq \mathcal{R} (\Gamma)$. On the other hand, if the reverse inclusion is not true, then we can find a polynomial $P \in k[\Omega]$ such that $P$ contains an indeterminate $x \in \Omega \setminus \Gamma$. Since any $g \in G_{\Gamma}$ fixes $P$, every element in the orbit $G_{\Gamma} \cdot x$ shall appear in $P$. But $G$ satisfies the strong amalgamation property, so this orbit is infinite, and hence $P$ have infinitely many indeterminates. This is absurd. Thus the reverse inclusion holds.

We have proved that $\mathcal{R}(\Gamma) = k[\Gamma]$. It is routine to check the second part of the description based on the explicit construction described in \cite[Secton 2]{LPY}.
\end{proof}

By Lemma \ref{structure sheaf}, the structure sheaf $\mathcal{R}$ over the atomic site $(\C^{\op}, J_{at})$ has the following description: $\mathcal{R}_n = k[x_1, \ldots, x_n]$ for every $n \in \mathbb{N}$, and a morphism $\epsilon \in \C([m], [n])$ gives the following ring homomorphism
\[
\epsilon^{\ast}: k[x_1, \ldots, x_m] \longrightarrow k[x_1, \ldots, x_n], \quad x_i \mapsto \epsilon(x_i).
\]

\begin{example}
Let $\Omega  = \mathbb{Z}+$, $G = \Sym(\Omega)$, $\tilde{\Omega}$ the Cartesian product of $d$ copies of $\Omega$, and $R = k[\tilde{\Omega}]$. Then the corresponding structure sheaf $\mathcal{R}$ is precisely $\mathbf{X}^{\FI, d}$ defined in \cite[Definition 2.17]{NR}.
\end{example}

Denote by $\PSh(\C^{\op}, J_{at}, \mathcal{R})$ the presheaf category and the sheaf category over the ringed site $(\C^{\op}, J_{at}, \mathcal{R})$. It has a family of generators $\{ P_n \mid n \in \mathbb{N} \}$ defined as follows:
\begin{itemize}
\item for $s \in \mathbb{N}$, the value of $P_n$ on $[s]$ is
\[
P_n([s]) = \bigoplus_{\epsilon \in \C([n], [s])} \mathcal{R}_n \epsilon
\]
with $\mathcal{R}_n \epsilon$ the free $\mathcal{R}_n$-module with the basis element $\epsilon$;

\item for a morphism $\pi: [s] \to [r]$ in $\C$, the corresponding map $P_n([s]) \to P_n([r])$ is defined via
\[
a \epsilon \mapsto \pi^{\ast}(a) (\pi \circ \epsilon), \quad a \in P_n([s]), \, \epsilon \in \C([n], [s]).
\]
\end{itemize}
This is indeed a family of generators for $\PSh(\C^{\op}, J_{at}, \mathcal{R})$: the case that $\C = \FI$ or $\C = \OI$ is proved in \cite[Definition 3.16 and Proposition 3.18]{NR}, and the reader can prove the conclusion for the other cases. It is also a family of generators for $\Sh(\C^{\op}, J_{at}, \mathcal{R})$ since each $P_n$ is the sheaf corresponding to the discrete $RG$-module $R(G/G_{\Gamma})$ where $\Gamma$ is a subset of $\Omega$ with cardinality $n$; see \cite[Subsection 3.2]{LPY} for more details.

\begin{definition}
An object $V$ in $\PSh(\C, J_{at}, \mathcal{R})$ or $\Sh(\C, J_{at}, \mathcal{R})$ is said to be \textit{finitely generated} if there is an epimorphism
\[
\bigoplus_{n \in \mathbb{N}} P_n^{c_n} \longrightarrow V
\]
such that the sum of multiplicities $c_n$ is finite; it is Noetherian if every subobject of $V$ is finitely generated as well.
\end{definition}

\begin{remark} \label{equivalent definition}
Another definition is: $V$ is finitely generated if there exists a finite set of elements $v_i \in V([n_i])$, $i \in [s]$, such that the only subobject of $V$ containing these elements is $V$ itself. These two definitions are equivalent, see \cite[Proposition 3.18]{NR} and \cite[Proposition 3.7]{LPY}.
\end{remark}

Now we are ready to prove the following result.

\begin{theorem}
Every finitely generated object in $\PSh(\C, J_{at}, \mathcal{R})$ or $\Sh(\C, J_{at}, \mathcal{R})$ is Noetherian.
\end{theorem}

\begin{proof}
Firstly we consider the category $\PSh(\C, J_{at}, \mathcal{R})$. For $\C = \FI$ or $\C = \OI$, this has been established in \cite[Theorem 6.15]{NR}. We can use the same strategy to prove the conclusion for the other three categories.

Explicitly, let $\D = \OI$, and $\C$ be $\BI$, $\CI$ or $\SI$. Then $\D$ is a subcategory of $\C$. Moreover, one can easily check the following fact: for any $m, n \in \mathbb{N}$ and any morphism $\epsilon \in \C([m], [n])$, there is a unique $\epsilon' \in \D([m], [n])$ and a unique $g \in \C([m], [m])$ such that $\epsilon = \epsilon' \circ g$. Therefore, for a fixed member $P_n^{\C}$ in the above family of generators of $\PSh(\C, J_{at}, \mathcal{R})$, its restriction to $\D$ is a direct sum of $|\C([n], [n])|$ copies of $P_n^{\D}$. Consequently, an object $V$ in $\PSh(\C, J_{at}, \mathcal{R})$ is finitely generated only if its restriction in $\PSh(\D, J_{at}, \mathcal{R})$ is finitely generated. The other direction is also true, following from the equivalent definition given in Remark \ref{equivalent definition} and the fact that $\D$ is a subcategory of $\C$.

We have proved that $V$ is finitely generated if and only if its restriction to $\D$ is finitely generated. Thus it is Noetherian if and only if so is its restriction to $\D$. The conclusion for $\PSh(\C, J_{at}, \mathcal{R})$ then follows from this observation and \cite[Theorem 6.15]{NR}.

Now we turn to $\Sh(\C, J_{at}, \mathcal{R})$. Let $W$ be a subobject of a finitely generated object $V$ in $\Sh(\C, J_{at}, \mathcal{R})$. We need to show that $W$ is finitely generated in $\Sh(\C, J_{at}, \mathcal{R})$. Note that we have a pair of adjoint functors
\[
\mathrm{inc}: \Sh(\C, J_{at}, \mathcal{R}) \longrightarrow \PSh(\C, J_{at}, \mathcal{R}), \quad \sharp: \PSh(\C, J_{at}, \mathcal{R}) \longrightarrow \Sh(\C, J_{at}, \mathcal{R})
\]
where $\mathrm{inc}$ is the inclusion functor and $\sharp$ is the sheafification functor. Via $\mathrm{inc}$, we can view $W$ as a subpresheaf of the presheaf $V$. Furthermore, we can assume that $V$ is finitely generated in $\PSh(\C, J_{at}, \mathcal{R})$. Indeed, since $V$ is a finitely generated sheaf, we have an epimorphism
\[
\varphi: \bigoplus_{n \in \mathbb{N}} P_n^{c_n} \longrightarrow V
\]
in $\Sh(\C, J_{at}, \mathcal{R})$ such that the sum of multiplicities $c_n$ is finite. Let $V'$ be the image of $\varphi$ in $\PSh(\C, J_{at}, \mathcal{R})$. Then $V'$ and $V$ represent the same object in $\Sh(\C, J_{at}, \mathcal{R})$ since $\Sh(\C, J_{at}, \mathcal{R})$ is an exact reflective localization of $\PSh(\C, J_{at}, \mathcal{R})$. Thus if necessary we can replace $V$ by $V'$, which is a finitely generated presheaf. By the result we just established, $V$ is a Noetherian object in $\PSh(\C, J_{at}, \mathcal{R})$, so $W$ is finitely generated in $\PSh(\C, J_{at}, \mathcal{R})$, and hence finitely generated in $\Sh(\C, J_{at}, \mathcal{R})$ by the following well known fact: if a morphism between two sheaves is an epimorphism in the presheaf category, then it is also an epimorphism in the sheaf category.
\end{proof}

Thus by the equivalence established in \cite[Theorem 2.2]{LPY}, we have:

\begin{corollary}
Suppose that the action of $G$ on $\Omega$ is highly homogeneous, and let $R = k[\Omega]$, the polynomial ring over a commutative Noetherian ring $k$. Then every finitely generated discrete $RG$-module is Noetherian.
\end{corollary}

\end{document}